\documentclass{amsart}
\usepackage{amssymb, amsfonts, amsmath}
\date{}
\usepackage{graphicx}
\usepackage{caption}
\usepackage[usenames]{color}

\newtheorem{ass}{Assumption}
\def \R{I\!\!R}

\input{tcilatex}
\input tcilatex

\begin{document}

\title[Neuronal Network]{An Interacting Neuronal Network with Inhibition: theoretical analysis and perfect simulation}
\author{Branda Goncalves }
\address{Laboratoire de Physique Th\'eorique et Mod\'elisation, CY Cergy Paris Universit\'e , CNRS UMR-8089, 2 avenue Adolphe-Chauvin, 95302 Cergy-Pontoise, France \\
 E-mail: branda.goncalves@outlook.fr}

\maketitle

\begin{abstract}
We study a purely inhibitory neural network model where neurons are represented by their state of inhibition. The study we present here is partially based on the work of Cottrell \cite{Cot} and Fricker et al. \cite{FRST}. The spiking rate of a neuron depends only on its state of inhibition. When a neuron spikes, its state is replaced by a random new state, independently of anything else and the inhibition state of the other neurons increase by a positive value.
Using the Perron-Frobenius theorem, we show the existence of a Lyapunov function for the process. Furthermore, we prove a local Doeblin condition which implies the existence of an invariant measure for the process.
Finally, we extend our model to the case where the neurons are indexed by $ \mathbb{Z}. $ We construct a perfect simulation algorithm to show the recurrence of the process under certain conditions. To do this, we rely on  the classical contour technique used in the study of contact processes, and assuming that the spiking rate lies on the interval $[ \beta_* , \beta^* ], $ we show that there is a critical threshold for the ratio $ \delta=  \frac{\beta_*}{\beta^* - \beta_*}$ over which the process is ergodic. 
\end{abstract}

\section{Introduction}

For the operation of a neural network, neurons excite and or inhibit each other.
Here, we study a model of a purely inhibitory neural network where neurons are represented by their inhibitory state. 
The study we present is partially based on the work of Cottrell \cite{Cot}. Her model consists of considering $ N $ interacting neurons described their state of inhibition. In her work, a neuron spikes when its state touches the value 0. When a neuron spikes, the state of inhibition of the other neurons increase by a non-negative deterministic constant $\theta.$ The spiking neuron immediately receives a random inhibition independent  of anything else. In Cottrell's work the state of inhibition is just the waiting time until the next spike.

In the present work we generalize Cottrell's model in several natural ways. Actually, in Cottrell's model, the next spiking time in the neural net is deterministic and we will lift this assumption.  A random spiking time is more realistic than deterministic  one since stochasticity is present all over in the brain functioning. Secondly, to allow formal general models we allow the state of inhibition to decrease at a general rate in between the successive spikes of the network while in Cottrell's work the drift of flow is equal to $-1.$

In the first part of this paper, we consider systems of $ N $ interacting neurons, in which any neuron can spike at any time. The spiking neuron takes a new random state of inhibition, and the others increase their inhibitory state by a deterministic quantity that we will call the inhibition weight, which depends on the distance between the spiking neuron and the "receiving" neuron, so that a neuron located far away of the spiking neuron is not impacted by the spike. The model thus presented obviously extends Cottrell \cite{Cot} and Fricker et al. \cite{FRST} in two ways: the spiking time is no more deterministic but it is random; the dynamic of the process is no more constant.

Firstly, we show the existence of a Lyapunov function that allows us to formulate a sufficient condition of non-evanescence of the process in the sense of Meyn and Tweedie \cite{MT}, i.e. a condition ensuring that the process does not escape at infinity.  To do so, we introduce a reproduction matrix $H $ and we suppose the spectral radius of $H $ is lower than $1.$ The eigenvector associated with the spectral radius of ${H}$ allows us to find a Lyapunov function for the process. 

Secondly, we study the recurrence of the process relying on Doeblin conditions which we establish for the embedded chain sampled at the jump times. We show the existence of an invariant probability measure for the process. We do this in the case the distribution of the new states has an absolutely continuous density and the jump rate is bounded.

In a second part, we consider the case where we have an infinite number of neurons indexed by $\mathbb{Z} $ (see Comets et al. \cite{CFF}, Galves and L\"{o}cherbach \cite{GL} and Galves et al.\cite{GLO}). In the work of Ferrari et al. \cite{FGGL}, considering an infinite system of interacting point processes with memory of variable length, the authors investigated the conditions for the existence of a phase transition using the classical contour technique, based on the classical work of Griffeath \cite{Gri} on a contact process. Following the idea of  Ferrari et al. \cite{FGGL} and Griffeath \cite{Gri}, we construct a perfect simulation algorithm that allows us to show the recurrence of the process.  Assuming that the spiking rate takes values in the interval $[ \beta_* , \beta^* ], $ we show that there is a critical threshold for the ratio $ \delta=  \frac{\beta_*}{\beta^* - \beta_*}$ over which the process is ergodic.

This paper is organized as follows. In section 2 we describe the model and study the law of the first jump time of the process. The Foster-Lyapunov and Doeblin conditions are discussed to find  non-evanescence criteria and to show the existence of the invariant probability measure of the process in section 3 which is our first main result. Finally, in Section 4, we present a perfect simulation algorithm and we simulate the law of the state of inhibition of a given neuron in its invariant regime.

\section{The model}

\subsection{ Description of the model }
In our paper, let us consider we have $N$ neurons that are related to each other. For all $ i \in \{1, \cdots, N \}, \  X_t^{i,N} $ describes the state of inhibition of neuron $i $ at time $t. $
When the neuron $i \in \{ 1, \cdots, N \}$ spikes,
\begin{enumerate}
\item[•] The current state of inhibition of neuron $i $ is replaced by a new value $Y^i$ independently of anything else with distribution $F^i. $

\item[•] The state of inhibition of any neuron $j \neq i $ is increased by a positive value $W_{i\to j}$ at time $t$.

\end{enumerate}
In between successive jumps of the system, each neuron $i $ follows the deterministic  dynamic  
\begin{equation*}
\overset{.}{x}_{t}^i=-\alpha_i \left( x_{t}^i\right) \text{, }x_{0}^i=x^i,
\end{equation*}
with $\alpha_i \left( x^i\right) $ continuous on $\left[ 0,\infty \right) $,
positive on $\left( 0,\infty \right) $ and non-negative on\textbf{\ }$\left[
0,\infty \right) $ and $x=(x^1, \cdots x^N).$  Let $\beta_i \left( x^i\right) $ be a continuous positive and decreasing rate function on $\left[ 0,\infty \right) $. 
We have taken $ \beta_i $ to be decreasing so that the larger $ x_t^i $ is, the lower its probability of spiking and the smaller $ x_t^i $ is, the higher its probability of spiking. 

We are thus led to consider the piecewise deterministic Markov process
(PDMP) $X_t^N = (X_t^{1,N}, \cdots, X_t^{N,N}) \in \mathbb{R}_+^N. $ For $i \in \{1, \cdots, N\}, $ the dynamic of $X_t^{i,N} $ is given by: 

\begin{multline} \label{eq:formule 1}
dX_{t}^{i,N} = -\alpha_i \left( X_{t-}^{i,N}\right) dt+ \int_{0}^{\infty } \int_{0}^{\infty }(y^i-X_{t-}^{i,N})\mathbf{1}_{\left\{ r\leq \beta_i  \left( X_{t-}^{i,N} \right) \right\} }M^i\left( dt,dr, dy^i\right) \\
				  +  \sum_{j\neq i } W_{j \to i } \int_{0}^{\infty }\int_{0}^{\infty }\mathbf{1}_{\left\{ r\leq \beta_j \left( X_{t-}^{j,N} \right) \right\} }M^j\left( dt,dr, dy^j\right),
\end{multline}
where $M^i $ is a random Poisson measure with intensity $ dt dr F^i ( dy ) $ and for all $i $, the $M^i $ are all independent.   This model extends that of Goncalves et al. \cite{GHL} in the multidimensional case. 

\begin{remark}
 For all $i \in \{ 1, \cdots N \} , $ $X_t^{i,N} $ can be interpreted as the inhibition state of the neuron $i $  at time $t $ and $W_{j \to i} $ as the inhibition weight of the neuron $j $ on the neuron $i.$
When $W_{i \to j} \leq0, $ we say that the neuron $i $ is excitatory for the neuron $j $ and when $W_{i \to j} \geq 0, $ we say that the neuron $i $ is inhibitory for the neuron $j.$ In our paper we are interested in the case where neuron $i $ is inhibitory for neuron $j $ i.e., $W_{i \to j} \geq 0. $
\end{remark}

\begin{remark}
 The formula (\ref{eq:formule 1}) is well-posed in the sense that there is non explosion of the process .   Since $ \beta_i ( X_s^{i,N} ) \leq \beta_i(0) $ for all $i$ we deduce that $ \int_0^t \beta_i ( X_s^{i,N }) ds < \infty $ whence the non explosion, that is, almost surely, the process has only a finite number of jumps within each finite time interval.
 
\end{remark}

The infinitesimal generator associated with this model is given by: 
\begin{multline}
 G^NV(x)= -\sum_{i=1}^{N} \alpha_i(x^i)\frac{\partial}{\partial{x^i}}V(x)+ \sum_{i=1}^{N} \beta_i(x^i)\int_0^{\infty}F^i(dy^i)[V(x+e^i y^i -e^i x^i + \sum_{j \neq i} e^j W_{i \to j})\\
- V(x)]
\end{multline}
where $V$ is a smooth function and $e^i$ is the $i-th $ unit vector. 

 In other words, at each jump of the process, a single neuron spikes. If it is neuron $i $ then its state is replaced by $Y^i$ and all other neurons receive the inhibition weight $W_{i\to j } \geq 0$ for any $ j\neq i.$

\subsection{First jump time}
Let $N_t^i $ be the counting process of successive jumps of neuron $i, $ that is, $$ N_t^i= \int_0^t \int_{\mathbb{R}_+} \int_{\mathbb{R}_+} \mathbf{1}_{ \{ r \leq \beta_i(X_s(x^i))\}}M^i(ds, dr, dy^i)  $$ and $ S_1^{i} $ the first jump time of neuron  $ i, $ so we have $$ S_1^{i} = \inf \{ t>0 | N_t^i = 1 \} \; \text{ and } \; \mathbb{P}( S_1^{i} >t)= e^{-\int_0^t \beta_i(x_s^i(x^i))ds }. $$

Let $ S_1 $ be the first jump time of the first neuron to jump, that is, $ S_1= \min_i S_1^{i}. $ For all  $ t>0, $ 
\begin{equation}\label{eq:premiersaut}
\mathbb{P}( S_1 >t ) = \mathbb{P}( \min_i S_1^{i} >t )= \prod_{i=1}^N \mathbb{P}( S_1^{i} >t) = \prod_{i=1}^N e^{-\int_0^t \beta_i(x_s^i(x^i))ds}. 
\end{equation}
Moreover, if $t < \min_{i} t_0(x^i) $ where   $$ t_0(x^i):= \int_0^{x^i}\frac{dy}{\alpha_i(y)}  $$  is the time for the neuron $ i $ hit  $ 0 $ starting from $ x^i, $ we can write by making a change of variables that is no longer valid after touching $0, $ that $$ \mathbb{P}( S_1 >t ) = \prod_{i=1}^N e^{ - [ \Gamma_i(x^i)-\Gamma_i(x_{t}^i(x^i)) ]}, $$

with $ \Gamma_i(x^i):= \int^{x^i} \gamma_i(y)dy $ and $ \gamma_i(x^i)= \beta_i(x^i)/ \alpha_i(x^i). $

\begin{ass}\label{ass1}
$ \Gamma_i(0)=-\infty $ for all $1 \le i \le N$.
\end{ass}

\begin{proposition}\label{eq:tempsfini}
\begin{enumerate}
\item[1.] Suppose  Assumption $ \ref{ass1} $  holds. Then  $ S_1 < \infty $ almost surely.

\item[2.] Suppose  Assumption $ \ref{ass1} $ does not hold. \\
 - If $ t_0(x^i) < \infty $ and $ \alpha_i(0)=0 $ then   $ S_1 < \infty $ almost surely if and only if $ \beta_i(0)>0 $ for all $i.$\\
-  If $ t_0(x^i) = \infty $ then $ \mathbb{P}(S_1 = \infty) >0 $ i.e. with a positive probability the first jump time is infinite. 
\end{enumerate}

\end{proposition}

\begin{proof}
Let  $ N $ be fixed and suppose Assumption $ \ref{ass1} $ holds.

 If $ t_0(x^i) =\infty $ and letting $ t $ tend to $ \infty $ in  (\ref{eq:premiersaut}) we have $$ \mathbb{P}(S_1 = \infty) = \prod_{i = 1}^N e^{-[\Gamma_i(x^i)-\Gamma_i(x_{\infty}^i(x^i))]} = \prod_{i = 1}^N e^{-[\Gamma_i(x^i)-\Gamma_i(0)]},  $$ since   $x_{\infty}^i(x^i) = 0. $ Then $  \mathbb{P}(S_1 = \infty) = 0 $ that is $ S_1 < \infty $ almost surely.

If $ t_0(x^i) < \infty $ and letting $ t \uparrow \min_i t_0(x^i) $ in (\ref{eq:premiersaut}), we obtain  $$ \mathbb{P}(S_1 \geq  \min_i t_0(x^i)) = \lim_{t\uparrow \min_i t_0(x^i)}  \mathbb{P}(S_1 > t )= \prod_{i = 1}^N  e^{-[\Gamma_i(x^i)-\Gamma_i(0)]}= 0 $$ implying that $ S_1 < \infty $ almost surely.

Suppose  Assumption $ \ref{ass1} $ does not hold. If  $\alpha_i (0) = 0 $ (this means that the flow of the process brings us to $0$ at most ) the time for the neuron $ i $ hit $ 0 $ starting from $ x^i $ is finite i.e  $ t_0(x^i) < \infty $ then it is obvious (by definition of $t_0(x^i) $) to see that it is enough that $ \beta_i(0) > 0 $ to have   $ S_1 < \infty $ almost surely. 

If  Assumption $ \ref{ass1} $ does not hold and $ t_0(x^i) = \infty $ then by making $t \to \infty $ in (\ref{eq:premiersaut}) we have $\mathbb{P}(S_1 = \infty) > 0 $ that is $ S_1 = \infty $ with a positive probability.
 
\end{proof}
We finish this section with a simulation of the process starting from some fixed initial configuration $ (x_0^1, \cdots, x_0^N). $
For this, we assume that for all $i $ the jump rate $\beta_i(x^i) $ is bounded and lower bounded, that is, $\beta_i(x^i) \in [\beta_*, \beta^*] $ for all $x^i>0, $ where $0< \beta_* < \beta^* < \infty.$ 

The following variables will be used to write our simulation algorithm.
\begin{enumerate}
\item[•] $ T $ is the time vector
\item[•] $ L $ is the label associated with $T.$ It will be $ \{sure\} $ or $\{uncertain\}$
\item[•] $ P =( P^1, \cdots, P^N) $ is the vector of states of the $ N $ neurons at a fixed instant 
\item[•] $ I $ is the vector which represents the number of the neuron which spikes.
\end{enumerate}

\textbf{Algorithm }
\begin{enumerate}
\item We set $T_1 \sim \exp(\beta^* N)$ \\
- $L_1 = \{uncertain\} $  with probability $\frac{\beta^* -\beta_*}{\beta^*} $  \\
-  $L_1 = \{ sure\} $ with probability $\frac{\beta_*}{\beta^*} $ 
\item We initialize the vector $P $ with the values $(x_0^1, \cdots, x_0^N )$ 
\item We choose $I_1 = k $ with probability $\frac{1}{N}$\\
- If $L_1 = \{ sure\}, $ \begin{equation*}
(a)
\begin{cases}
P^k \sim F^k, \\ P^i \gets x_{T_1}^i(P^i)+ W_{k \to i}
\end{cases}
\end{equation*}
-If $L_1 = \{uncertain\} $ we accept the jump with probability $$p= \frac{\beta_k(x_{T_1}^k(P^k))-\beta_*}{\beta^*-\beta_*}, $$ and we apply $(a).$\\
-Else $P^j \gets x_{T_1}^j(P^j), \;\; \forall j \in \{1, \cdots, N \}. $ 
\item We update the vector $P $ and start the procedure again from $(1).$
\end{enumerate}

We stop the procedure after a fixed finite number $ n $  of iterations.

We plot in the following figure a typical trajectory of $ X_t^{i,N} $ with $N=2 $ neurons.

\begin{figure}[h]
    \begin{minipage}[c]{.46\linewidth}
        \centering
        \includegraphics[scale=0.4, height=6cm, width=7.5cm]{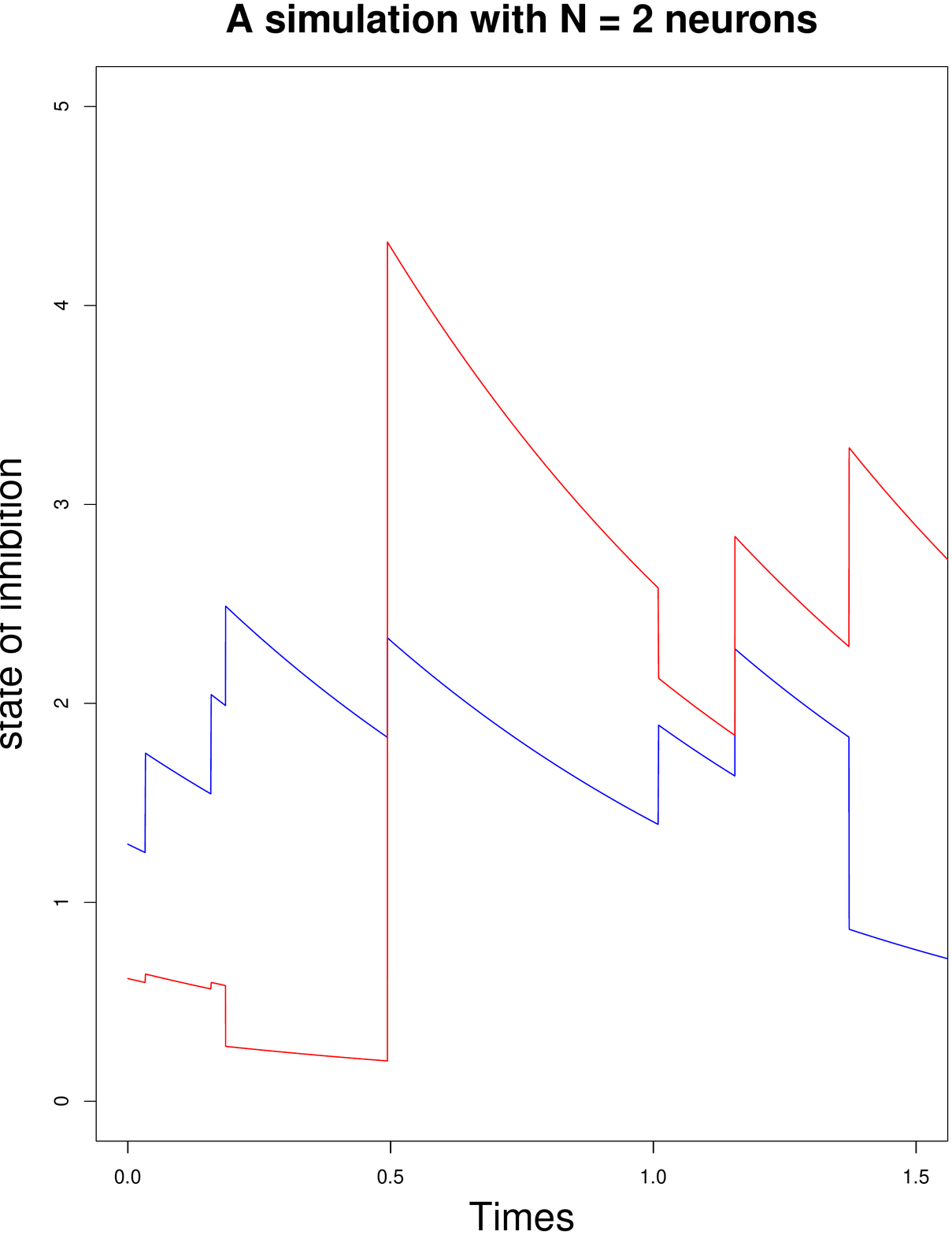}
        \caption{}
    \end{minipage}
    \hfill%
    \begin{minipage}[c]{.46\linewidth}
        \centering
        \includegraphics[scale=0.4, height=6cm, width=7.5cm]{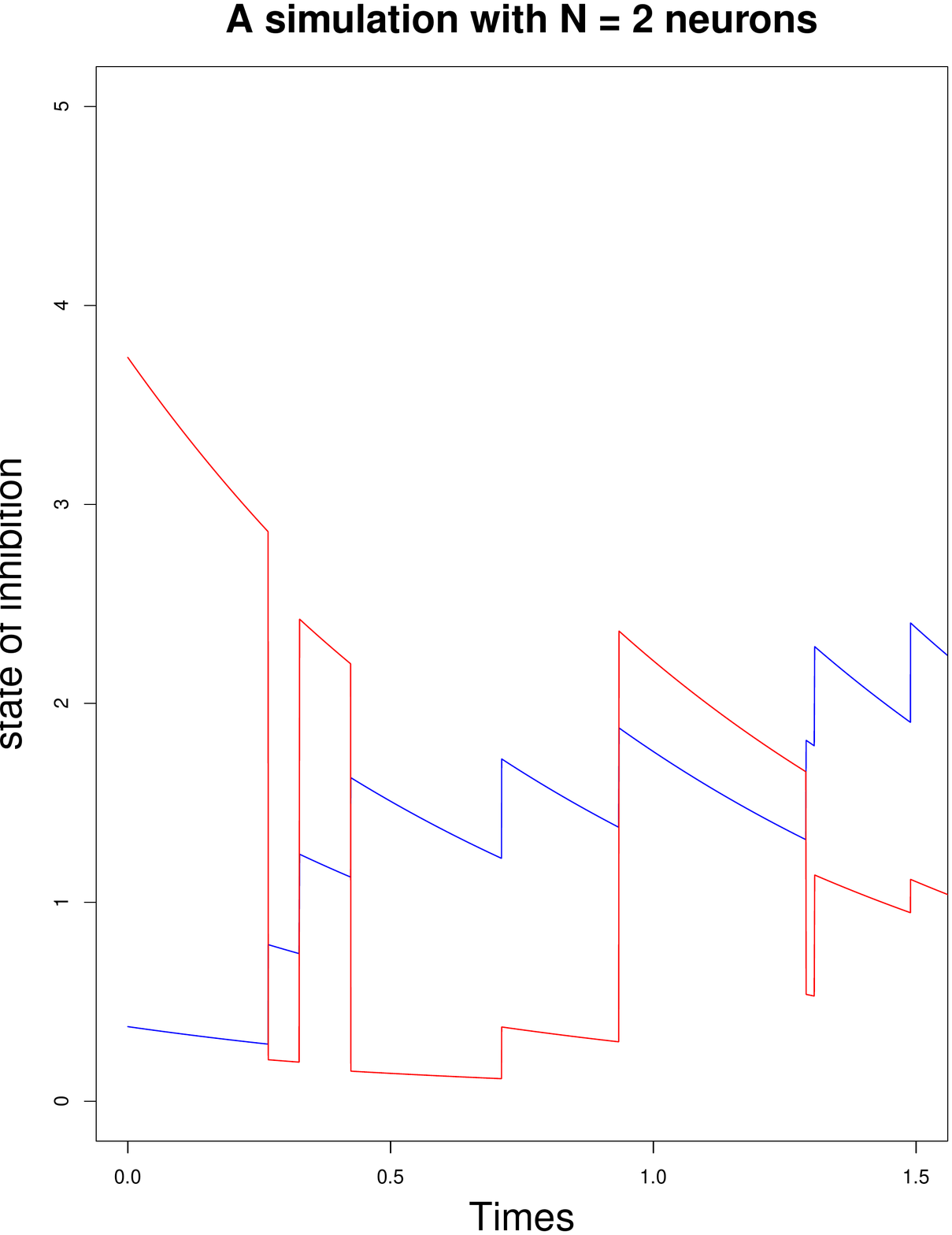}
        \caption{}
    \end{minipage}
\end{figure}

In both figures $ N=2 $ neurons and $n =50 $ iterations. $ \alpha (x)= x, \beta(x) = 3+ \mathbf{1}(x \leq 2).$  In the figure on the left, $W_{j \to i} = i/N $ for all $j\neq i $ and in the figure of right, $W_{j \to i } = 1/2 $ for all $j\neq i $.

\section{Foster-Lyapunov and Doeblin conditions }

In this section, we want to find conditions of non-evanescence of the process and show the existence of an invariant probability measure of the process.

\subsection{Foster-Lyapunov condition }

We suppose that $ \gamma_i := \beta_i/\alpha_i $ is bounded and we define $ W $ the matrix of inhibition weight  by $ W_{ij} := W_{j \to i} , \ i\neq j \text{ and }  W_{ii} = 0.$ 

It is further assumed that the matrix $W $ is irreducible in the sense that there exists an integer $p>0 $ such that $W^p >0. $  We introduce the reproduction matrix
$$ 
 H_{ij } = W_{j \to i }\| \gamma_i\|_{\infty} ,\  i \neq j, \ H_{ii} = \| \gamma_i\|_{\infty} \int_0^\infty y^i F^i(dy^i) $$ 
which is also irreducible.

Suppose that $$ \rho (H)  < 1 $$
where $\rho (H) $ is the largest eigenvalue of $H $ that is the spectral radius of $H. $ 
Then (Perron Frobenius) there exists a left eigenvector $ \kappa $ associated to this eigenvalue $ \rho, $ that is, for all $i, $ 
$$ \sum_j  \kappa_j H_{j i } = \rho \kappa_i .$$ 
On the other hand, put  $m_i = \kappa_i\| \gamma_i\|_{\infty} $ and et $ V(x) = \sum m_i x^i.$

Finally, let $ V: \mathbb{R}_+^N \to \mathbb{R} $ such that $ V(x)=\sum_{i=1}^N m_i x^i $ and we recall that the infinitesimal generator is given by: 
\begin{multline*}
G^NV(x)= -\sum_{i=1}^{N} \alpha_i(x^i)\frac{\partial}{\partial{x^i}}V(x)+ \sum_{i=1}^{N} \beta_i(x^i)\int_0^{\infty}F^i(dy^i)[V(x+e^i y^i -e^i x^i + \sum_{j \neq i} e^j W_{i \to j})\\
- V(x)]
\end{multline*}
So by replacing $ V $ by its expression in the infinitesimal generator $ G^NV(x) $ we have:
\begin{multline*}
 G^NV(x) = -\sum_{i=1}^{N} \alpha_i(x^i)m_i+ \sum_{i=1}^{N} \beta_i(x^i)\int_0^{\infty}dF^i(y^i)[\sum_{j=1, j \neq i}^N (W_{i \to j}+x^j ) m_j + y^im_i - \sum_{j=1}^N x^j m_j]  \\
 = -\sum_{i=1}^{N} \alpha_i(x^i)m_i+ \sum_{i=1}^{N} \beta_i(x^i)(m_i \int_0^{\infty} y^i F^i (dy^i) + \sum_{j\neq i } W_{i \to j } m_j )- \sum_{i=1}^{N}  \beta_i(x^i)x^i m_i
\end{multline*}

Then, since $ - \beta_i (x^i ) x^i \le 0,$  

\begin{multline*}
G^N V ( x) \le  - \sum_i \alpha_i ( x^i ) m_i + \sum_i \beta_i (x^i ) ( m_i \int_0^{\infty} y^i F^i (dy^i) + \sum_{j\neq i } W_{i \to j } m_j ) \\
=  - \sum_i \alpha_i ( x^i ) \left( m_i - \gamma_i ( x^i )  \left[ m_i c_i  H_{ii} + \sum_{ j \neq i }c_j H_{ji } m_j  \right] \right) \\
=  - \sum_i \alpha_i ( x^i ) \left( m_i - \gamma_i ( x^i )  \left[ \kappa_i  H_{ii} + \sum_{ j \neq i } \kappa_j H_{ji }  \right] \right) \\
=  - \sum_i \alpha_i ( x^i ) \left( \| \gamma_i\|_\infty \kappa_i - \gamma_i ( x^i )  \rho \kappa_i \right) \\
=   - \sum_i \alpha_i ( x^i )  \| \gamma_i\|_\infty \kappa_i  \left( 1 - \frac{ \gamma_i ( x^i ) }{ \| \gamma_i\|_\infty} \rho \right) . 
\end{multline*}

\begin{definition}
We call the process non evanescent if there exists a compact $K$ such that for all $x, $ $\mathbb{P}_x-$ almost surely, $\limsup_t 1_{K} (X_t) =1.$ 
\end{definition}

\begin{proposition}
If $\rho <1, $ then the process is non-evanescent. 
\end{proposition}

\begin{proof}
$V(x) $ defined in $G^NV(x) $ above is a norm-like function and $G^NV(x) <0. $
The theorem (CD0) of Meyn and Tweedie \cite{MT} implies the result.
\end{proof}

%Pour la these??
% \begin{remark}
%We can refine the above conditions as follows. Let $K$  be a compact set.
%Then we replace the above definition of $ H_{ii} $ by 
%$$ H_{ii} = \| \gamma_i\|_\infty \sup \{ ( \int y^i F^i ( dy^i ) - x^i )_+ : x \in K^c, \beta_i ( x^i ) > 0 \} $$
%and suppose that $ \rho (H) < 1.$  As a consequence, we obtain, with the same calculus, that on $ K^c,$ $$ G^N V (x) < 0 . $$ 
%\end{remark}

\begin{example}{ \bf (Mean-field interaction)} Suppose we have $ N $ neurons.
We suppose also that $ \gamma_i = \gamma, $  where $\gamma $ is bounded and  $  F^i=F, \, W_{j \to i}= \theta $ for all $i. $ In this case the reproduction matrix is $$ H_{ij}= \theta \| \gamma \|_{\infty}, \, i \neq j, \, H_{ii} = \| \gamma \|_{\infty} \mathbb{E}(Y). $$  
Suppose $\rho(H) $ is the spectral radius of $H. $ Then, $ \rho(H) = \| \gamma \|_{\infty}(\mathbb{E}(Y) + (N-1) \theta ) $ and its associated eigenvector is $\kappa = ( 1, \cdots, 1). $ The condition $\rho(H) <1 $ is therefore equivalent to $ \| \gamma \|_{\infty} (\mathbb{E}(Y) + (N-1) \theta ) <1.$ For $\theta = \varepsilon/N $ with $\varepsilon >0 $  the condition $\rho(H) <1 $ becomes  $ \| \gamma \|_{\infty} (\mathbb{E}(Y) + \varepsilon) < 1. $
\end{example}

\begin{example} {\bf (Torus)}
Suppose we have $  N \geq 3 $  neurons such that each neuron interacts with its two nearest neighbors (its left and right neighbors). Neuron  $ 1  $ interacts with neuron  $ 2 $ and neuron  $ N $ . Neuron $ N $ interacts with neuron $ N-1 $ and neuron $ 1, $ so we have a torus.

We suppose for all $i, $ we have $ \gamma_i = \gamma, $ bounded and  $ \, F^i=F. $ $ W_{j \to i}= \theta $ for all $ j \in \{i+1, i-1\} $ and $ W_{j \to i}= 0 $ if $ j \neq \{i+1, i-1\}. $  In this case the reproduction matrix is 
\begin{equation*}
H_{ij} = 
\begin{cases}
\theta \| \gamma \|_{\infty} & \text{ , if }  i \neq j, \, j \in \{i+1, i-1\} \\
0  & \text{ , if }   i \neq j, \, j \neq \{i+1, i-1\} \\
\| \gamma \|_{\infty} \mathbb{E}(Y)  &\text{ , if } i=j. 
\end{cases}
\end{equation*}
If $\rho (H) $ is the spectral radius of $ H $ then $ \rho(H) = \| \gamma \|_{\infty} ( \mathbb{E}(Y) + 2 \theta ) $ and its associated eigenvector is $\kappa = ( 1, \cdots, 1). $  The condition $\rho(H) <1 $ is equivalent to $ \| \gamma \|_{\infty} (\mathbb{E}(Y) + 2 \theta  )<1.$ 
\end{example}

\subsection{Doeblin condition }

Let $ S_0 < S_1 < \cdots < S_n<\cdots $ be the instants of successive jumps of the $ N $ neurons. It is obvious that the embedded chain  $ Z_n := X_{S_n} $ is a Markov chain.  
Let $ I_n $ be the index of the neuron which jumps at time $ S_n. $ 

\begin{proposition} Suppose that the assumptions of proposition \ref{eq:tempsfini} hold. Then,
 $ (Z_n , I_n) $ is a Markov chain and its  transition $Q(x, dy)$  is given by:
\begin{multline}
\mathbb{P}(Z_n \in dy, I_n=j | Z_{n-1}=x, I_{n-1}=i) = \int_0^{\infty}ds e^{-\int_0^s dl \sum_{i=1}^N \beta_i(x_l^i(x^i)) } \beta_j(x_s^j(x^j)) \\
 \times  \int F^j(du) \delta_{(x_s^1(x^1)+W_{j \to 1}, \cdots, x_s^{j-1}(x^{j-1})+W_{j \to j-1}, u,\; x_s^{j+1}(x^{j+1})+W_{j \to j+1}  , \cdots, x_s^N(x^N)+W_{j \to N} )} (dy). 
  \end{multline}
\end{proposition}

 \begin{theorem}\label{theo:petiteset}
Suppose for all $1 \le i \le N, \, \alpha_i \in \mathcal{C}^1 $ and there exists a compact set $K \subset ]0, \infty[^N $ such that for all $ x \in K, $ for all $ 1 \le i \le N, $ $  \beta_i(x^i+\sum_{j=1}^{i-1} W_{j \to i}) >0  .$ Moreover we suppose that $F^i(dy) $ is absolutely continuous and $ \| \beta_i\|_\infty < \infty $ for all $ i.$ 
Then there exist $ d \in (0, 1 ) $ and 
a probability measure $ \nu $ on $ ( \R_+^N , {\mathcal{B} }( \R_+^N) ) ,$ such that 
\begin{equation}
 Q ^N(x, dy ) \geq d 1_K (x) \nu (dy) 
\end{equation}
where $Q $ is the transition operator of embedded chain $Z_n = X_{S_n} $ and $Q ^N $ is its $N-$th  iterate.
 \end{theorem}
 
To prove the above result we fix any deterministic sequence $s_1 <\cdots <s_N. $ In the sequel we shall work on the event $S_1=s_1, \cdots, S_N= s_N, I_1=1, \cdots, I_N=N $ and $Y_1=y_1, \cdots, Y_N=y_N. $
This means that the jumps are ordered such that neuron $1 $ jumps before neuron $2 $ and etc. 
Let $y=(y^1, \cdots, y^N) $ where $y^i $ is the new state of inhibition of neuron $i $ after the spike.

Let $t_k= s_k-s_{k-1} $ for all $1\leq k\leq N $  the inter jump times of the $N $ neurons which implies that $s_k= t_1+\cdots +t_k. $

Conditionally on this event, let $ \Psi_{s_N} $ be the vector of states of the process at time $s_N. $ We can define $ \Psi_{s_N}$ as a function of the states $ y_1, \cdots, y_N $ such that $ \Psi_{s_N}: \mathbb{R}^N \to  \mathbb{R}^N $ is given by:
\begin{equation*}
\Psi_{s_N}^k(y) =
\begin{cases}
  \psi_{t_N}^{k,N} \circ \cdots \circ \psi_{t_{k+1}}^{k,k+1}(y^k) & \text{, if $ 1\le k<N $ } \\ y^N & \text{, if $ k=N $ }
\end{cases}
\end{equation*}
where for all $l \neq k, $
\begin{equation}\label{position des sauts}
 \psi_{s}^{k,l}(u)=x_{s}^k(u)+W_{l \to k} 
\end{equation}
 and $x_{s}^k(u) $ means  the solution of the deterministic dynamic $ \overset{.}{x}_s^k  = -\alpha_i(x_s^k), \ x_0^k=u. $
   
\begin{remark}\label{eq:deriveePsi}
In the definition of  $\Psi_{s_N}^k (y) ,$  we note that it depends only on $y^k. $ Therefore we have for all $ i \neq j, $   $$ \frac{\partial \Psi_{s_N}^i}{\partial y^j}=0.  $$
\end{remark}

\begin{proposition}\label{eq:jacobian}
For all $1 \le k \le N $ let $\alpha_k $ be a globally Lipschitz function. For all $y \in \mathbb{R}_+^N ,$ there exists an open neighborhood $\mathcal{B} $ of $y $ such that $ \Psi_{s_N}: \mathcal{B} \to \Psi_{s_N}(\mathcal{B}) $ is a local diffeomorphism.
\end{proposition}

\begin{proof}

 Let $ J_{\Psi_{s_N}(y)} $ be the Jacobian matrix of  $ \Psi_{s_N}(y) .$ Using the remark \ref{eq:deriveePsi} we have : 
$$ \det(J_{\Psi_{s_N}(y)})= \det
\begin{pmatrix}
 \frac{\partial \Psi_{s_N}^1(y)}{\partial y^1} & \cdots &\frac{\partial \Psi_{s_N}^1(y)}{\partial y^N} \\\\ \vdots & \ddots & \vdots \\\\ \frac{\partial \Psi_{s_N}^N(y)}{\partial y^1} & \cdots & \frac{\partial \Psi_{s_N}^N(y)}{\partial y^N}
\end{pmatrix} =\det
\begin{pmatrix}
 \frac{\partial \Psi_{s_N}^1(y)}{\partial y^1} & 0 &\cdots &0 \\ 0 & \ddots & & \vdots \\ \vdots & & \ddots & 0 & \\ 0 & \cdots & 0 & \frac{\partial \Psi_{s_N}^N(y)}{\partial y^N}
\end{pmatrix}. 
$$
We obtain
$\det( J_{\Psi_{s_N}(y)} )\neq 0 $ if and only if $\prod_{j=1}^N\frac{\partial \Psi_{s_N}^j(y)}{\partial y^j}  \neq 0 $ that is $ \frac{\partial \Psi_{s_N}^j(y)}{\partial y^j} \neq 0, \forall 1\le j \le N.$ 
It is obvious to see that $\frac{\partial \Psi_{s_N}^N(y)}{\partial y^N} = 1.$

For all $1\le j \le N-1, $ we have:
\begin{equation}\label{formule6}
 \frac{\partial \Psi_{s_N}^j}{\partial y^j}(y)= \prod_{i=1}^{N-j} \exp \left({-\int_{s_{N-i}}^{s_{N-(i-1)}} \alpha'_j  \left(x_s^j \left(\Psi^j_{s_{N-i}}(y) \right) \right)ds} \right) \neq 0. 
 \end{equation}

It means that $ |\det(J_{\Psi_{s_N}(y)})| \neq 0 $ then $\Psi_{s_N}(y) $ is a local diffeomorphism. 

 Localizing, we may therefore conclude that for each $ y $ there exists $ \mathcal{B} $ such that  $\Psi_{s_N} :  \mathcal{B} \to \Psi_{S_N}(\mathcal{B}) $ is a diffeomorphism.
\end{proof}

\begin{proof}[Proof of theorem \ref{theo:petiteset}]
Let $ \varepsilon >0 $ fixed. We will work on the event $$E=\{ S_1\leq \varepsilon, \cdots, S_{n+1}-S_n \leq \varepsilon,\forall n<N: ( I_1, \cdots, I_N ) = ( 1, \cdots, N)  \}. $$ In particular, on $ E, $ the index $ I_n $ of the $ n-$th neuron is equal to $ n $ for all $ n \in \{ 1, \cdots, N\}. $ 

Knowing that the first jump takes place at time $ S_1=s_1, $  the probability that the index $ I_1 $ of the first jump is equal to $ 1 $ is given by: 
 $$\mathbb{P}(I_1= 1 | S_1=s_1 )= \mathbb{P}(S_1^{1} < S_1^{j}, \forall j \neq 1 ) = \frac{ \beta_1(x_{s_1}^1(x^1))}{  \sum_{j=1}^N  \beta_j(x_{s_1}^j(x^j)) }.  $$ 
 We want to compute, $\mathbb{P}(I_1=1, \cdots, I_N=N | S_1=s_1, S_2=s_2, \cdots, S_N=s_N  ). $ To obtain a compact formula,
using formula (\ref{position des sauts}) we define \begin{equation*}
 \phi_{j}^k(x^k, y^k, s_1, \cdots, s_N) = 
 \begin{cases}
 \psi_{t_{j}}^{k,j} \circ \cdots \circ \psi_{t_{k+1}}^{k,k+1 } (y^k), & \text{ if $ 1 \le k \le j-1 $} \\
 \psi_{t_{j}}^{k,j} \circ \cdots \circ \psi_{t_{1}}^{k,1 } (x^k), &  \text{ if $ j \le k  \le N $} 
 \end{cases}
 \end{equation*}
 giving the states of neuron $k $ at time $S_j$ depending on whether neuron $k $ jumped before or after time $S_j$.\\
Let $$ x_j^k = x_{t_j}^k(\phi^k_{j-1}(x^k, y^k, s_1, \cdots, s_N)$$ be the state of neuron $k $ before the $j-th $ jump.
 We know that as long as neuron $k $ has not yet jumped, it receives each time a quantity $ W_{j \to k} , \ \forall j \neq k $ from the other neurons that jumped before it. So knowing all the jump times where other neurons jumped, we have:
 \begin{multline*}\label{proba de E}
 \mathbb{P}(I_1=1, \cdots, I_N=N | S_1=s_1, S_2=s_2, \cdots, S_N=s_N  ) = \\
 \frac{ \beta_1(x_{s_1}^1(x^1))}{ \sum_{i=1}^{N}  \beta_i(x_{s_1}^i(x^i)) } \int_{\mathbb{R}_{+}^{N-1} }\frac{ \prod_{i=2}^{N}  \beta_i(x_{i}^i)}{  \prod_{i=2}^{N} (\sum_{k=1}^{N}\beta_ k( x^k_{i}))} \prod_{k=1}^{N-1}\mathbb{P}(Y^k \in dy^k).
\end{multline*}

For any Borel subset $B $ of $\mathbb{R}^N $ we have 
 \begin{multline*}
  Q^N(x, B) \geq \mathbb{P}_x(Z_N \in B, E)= \int_{[0,\varepsilon]^N}dt_1\cdots dt_N  \int_{\mathbb{R}^N}F^1(dy^1) \cdots F^N(dy^N) \\
 \times (\prod_{k=1}^N \beta_k(x_k^k))e^{-\int_0^{s_N}\sum \beta_k(\Psi_t^k(y))dt} \mathbf{1}_B(\Psi_{s_N}(y)).
 \end{multline*}
 
 Remark that on the event $E, $ $x_k^k \leq x^k+ \sum_{j=1}^{k-1} W_{j\to k}.$ 
Recall $\beta_k $ is decreasing function and let $\mu_k = \inf_k \{ \beta_k( x^k+ \sum_{j=1}^{k-1} W_{j\to k}): x \in K \} $ the  lower-bound on $K $ of $ \beta_k (x^k + \sum_{j=1}^{k-1} W_{j\to k}).$  
 
 Using the fact that $\| \beta_i\|_\infty <\infty$ for all $i, $ let $c = (\prod_{k=1}^N \mu_k )e^{-N\| \beta_i\|_\infty N\varepsilon}. $ 
 Then we have 
 \begin{equation}\label{eq:majorationQn}
 Q^N(x, B) \geq c \int_{[0,\varepsilon]^N}dt_1\cdots dt_N  \int_{\mathbb{R}^N}F^1(dy^1) \cdots F^N(dy^N) \mathbf{1}_B(\Psi_{s_N}(y)). 
\end{equation}  
 
Following the arguments of Bena\"{i}m et al. \cite{BLMZ}, for any $t^*  \leq N\varepsilon, $ there exists a ball $B_r(t^*) $ of radius $r, $ of center $t^* $ and an open set $I \subset \mathbb{R}^N $ such that we can find for all $ s_N\in B_r(t^*), $ an open set $W_{s_N} \subset \mathbb{R}^N : $
\begin{equation*}
\tilde{\Psi}_{s_N}: 
\begin{cases}
W_{s_N} \to I \\ 
y \mapsto \Psi_{s_N}(y)
\end{cases}
\end{equation*}
is a diffeomorphism (see Bena\"{i}m et al. \cite{BLMZ}, Lemma 6.2). In the above formula, $\tilde{\Psi}_{s_N} $ denotes the restriction of $\Psi_{s_N} $ to $W_{s_N}. $
This allows us to apply the theorem of a change of variables in the inequality (\ref{eq:majorationQn}).

 $ \alpha'_j  \left(x_s^j \left(\Psi^j_{s_{N-i}}(y) \right) \right) $ is upper bounded since $ \alpha_j $ is a global Lipschitz function. Then, for all $1 \le j \le N-1 $ we obtain: $$ \frac{\partial \Psi_{s_N}^j}{\partial y^j}(y) \leq  \exp (N-j)\varepsilon \| \alpha_j ' \|_{\infty}.$$

Then, $ \forall y\in W_{S_N}, c \,| \det(J_{\Psi_{s_N}(y)}) |^{-1} \geq c'>0 $ and 
the inequality (\ref{eq:majorationQn}) becomes :   
\begin{eqnarray*}
Q^N(x, B) &\geq &  c \int_{[0,\varepsilon]^N}dt_1\cdots dt_N \int_{\mathcal{B}}F^1(dy^1) \cdots F^N(dy^N) \mathbf{1}_B({\Psi}_{s_N}(y))\\
&\geq &  c'  \int_{B_{r}(t^*)}dt_1\cdots dt_N \int_{W_{s_N} \cap \mathcal{B} }dy \mathbf{1}_B(\tilde{\Psi}_{s_N}(y)) | \det(J_{\tilde{\Psi}_{s_N}(y)}) | \\
&\geq & c'  \lambda(B_{r}(t^*)) \int_I  \mathbf{1}_B(x)dx \, = \, d \mathbf{1}_B(x) \nu(I)  
\end{eqnarray*}
where $ d=  c'  \lambda(B_{r}(t^*)) $ with $\lambda(B_{r}(t^*)) $ the Lebesgue measure of the ball $B_{r}(t^*) $ and $\nu(I) $ the uniform measure of $I.$

 \end{proof}
 
\begin{corollary}
If for all $k \le N,  $ $\beta_k $ is strictly lower-bounded and bounded, then the process is recurrent. \end{corollary}

\begin{remark}
When $ \beta_k $ is strictly lower-bounded and bounded, we can notice that the lower bound obtained in theorem \ref{theo:petiteset} holds on the whole state space $\mathbb{R}_+, $ that is, without $1_K. $ This allows us to have the global lower bound $ Q^N(x, dy) \geq d \nu(dy) $ and thus the uniform ergodicity of the process.
\end{remark}

\section{Perfect simulation}
In this section, we consider a framework with an infinity of neurons indexed by $\mathbb{Z}.$
We want to build a perfect simulation algorithm to show in another way the recurrence of our process under certain conditions. Let  $\mathcal{V}_{. \to i } = \{ j: W_{j \to i} \neq 0\} $ and $\mathcal{V}_{i \to . } = \{ j: W_{i \to j} \neq 0\} $ be the incoming and out-coming neighborhoods of the neuron $ i $ (see Comets et al. \cite{CFF} and  Galves and L\"{o}cherbach \cite{GL}).

We consider the case where each neuron has a finite number of neighbors. 

We assume that for all $i $ the jump rate $\beta_i(x^i) $ is bounded, that is, $\beta_i(x^i) \in [\beta_*, \beta^*] $ for all $x^i>0, $ where $0<\beta_* < \beta^* < \infty. $

The following variables will be used to write the perfect simulation algorithm:

\begin{enumerate}
\item[-] $ T $ is the time vector
\item[-] $ P $ is the matrix of states where each row of this matrix represents the different states of the $ N $ neurons at a fixed instant
\item[-] $ I $ is the vector which represents the index of the neuron which spikes.
\end{enumerate}

We fix a neuron $ i \in \mathbb{Z} $ and in what follows we are interested in finding the state of $i $ at time $0 $ in the stationary regime, that is, assuming that the process starts from $-\infty. $ To do so we explore the past in order to determine all sets of indices and times which affect the value of neuron $i $ at time $0.$ The set of all such couples $(j,s ) $ will be called the clan of ancestors of neuron $i $ (see  Galves and L\"{o}cherbach \cite{GL}, Galves et al. \cite{GLO}). The clan of ancestors is a process that evolves in time by successive jumps. We start with $C_0^i = \{ i \} $ and in the following we will define the updates of this process at the time of the jumps. More precisely we do the following: 

\begin{enumerate}
\item[-] We simulate , $\forall \; l \in \mathbb{Z}, \;\; N_t^{l,s} \text{ and } N_t^{l,p} $ two Poisson processes with respective intensities $\beta_* $ and $\beta^*-\beta_*. $ The jump times of $ N_t^{l,s} $ and $N_t^{l,p} $ are respectively $T_n^{l,s} $ and $T_n^{l,p} $ for the neuron $l $ after $n $ jumps.\\

\item[-] Let $ i \in \mathbb{Z} $ be fixed  and   $ T_1= \inf \{ T_1^{j, s}, T_1^{j, p} >0: \; j  \in \mathcal{V}_{. \to i} ,\;  T_1^{i,s}>0 \} $ where $ \mathcal{V}_{ . \to i}$ is the incoming neighborhood of $i. $ 

-If $T_1= T_1^{j,s}  , $ we set $C_{T_1}^i  = \{ i\}  $ and $I_1= j.$

- If $T_1= T_1^{j,p}  , $ we set $C_{T_1}^i  = \{ i, j\} $  and we set $I_1= j .$

- If $T_1= T_1^{i,s}, $ we set $C_{T_1}^i = \emptyset $ and we stop the algorithm. In this case we set $I_1= i.$\\

\item[-] Suppose $T_n $ is the $n-th $ jump time of $C_{T_n}^i $. Then, $$ T_{n+1}=\inf\{T_m^{j, s}, T_m^{j, p}>T_n: \exists l \in  C_{T_n}^i,  \; j  \in \mathcal{V}_{. \to l}, \;  T_m^{k,s}>T_n, \; k\in  C_{T_n}^i \}. $$

- If $T_{n+1}= T_m^{j, s}  $ we set $ I_{n+1} = j $ and then  $C_{T_{n+1}}^i=C_{T_n}^i $ 

- If $T_{n+1}= T_m^{j, p}  $ we set $ I_{n+1} = j $ and then  $C_{T_{n+1}}^i=C_{T_n}^i \cup \{j\} $ 

- If $T_{n+1}= T_m^{k,s} $ we set $ I_{n+1} = k $ and then $ C_{T_{n+1}}^i=C_{T_n}^i \setminus \{k\} $ where $ k\in  C_{T_n}^i. $

\end{enumerate}

We stop the procedure at time $T_{stop}^i = \inf\{ t : C_t^i = \emptyset\}. $ To ensure that the algorithm stops it will be necessary to find a criterion so that $ T_{stop}^i < \infty.$ This will be done in Theorem \ref{theo:val critique} below.
The above algorithm is called the backward procedure.

%{\color{red} Pour la these.  The graph below is a typical example of the simulation of the two Poisson processes $N_t^{l,s} $ and $N_t^{l,p} $ of respective intensities $\beta_* $ and $\beta^*-\beta_*.$
%\begin{center}
%\includegraphics[scale=0.4, height=6cm, width=15cm]{RPoissProcesses}
%\end{center}
%The process $N_t^{l,s} $ is represented in black and the process $N_t^{l,p} $ in red. In this example $N= 6 $ neurons, $\beta_* = 3 $ and $\beta^*-\beta_* = 1.$ }

In the following we will write a forward procedure of the process in case where each neuron has a finite number of neighbors.  \\
For this we define: $$ N_{stop}^i = \inf\{ n >0: C_{T_n}^i= \emptyset \}, \    \bar{\mathcal{C}}^i = \cup_{n=0}^{{N}^i_{stop}} C_{T_n}^i \text{ and } \partial_{ext}(C_t^i) = \{ j \notin C_t^i : \exists k\in C_t^i,  j\to k  \} $$ 
where $ N_{stop}^i $ is the number of steps of the backward procedure, $ \bar{\mathcal{C}}^i  $ is the union of all clans of ancestors up to $N^i_{stop}$ and $ \partial_{ext}(C_t^i) $ is the set of neurons not belonging to the clan of ancestor  of neuron $i $ but having an interaction with at least one neuron in the ancestor clan of neuron $i.$

In this algorithm, we will rely on the a priori realizations of the processes  $  N_t^{i,s} \text{ and } N_t^{i,p}.  $

\textbf{Algorithm (forward procedure)}

\begin{enumerate}
\item We initialize the set of sites for which the decision to accept  can be made by 
$$ \mathcal{S}^i= \{ (I_m, T_m) \in \bar{\mathcal{C}}^i \times \mathbb{R}_+,  C_{T_m}^{I_m} = \emptyset\} $$

\item[] For $n =  N_{stop}^i  \text{ we have } P_n^{I_n} \sim F^{I_n} $. Starting from $n \to n-1 :$

\item If  $ (I_{n-1}, T_{n-1} ) \in \mathcal{S}^i $ then $ P_{n-1}^{I_{n-1}} \sim F^{I_{n-1}} .$ \\
- If for $ j \in \mathcal{V}_{  I_{n-1} \to . }, $ we have $ j \in C_{T_{n-1}}^i $ then  $$ P_{n-1}^j = x_{T_n-T_{n-1}}^{j } (P_{n}^{j }) + W_{ I_{n-1}  \to j }$$
- If for $ j \notin \mathcal{V}_{  I_{n-1} \to . }, $ we have  $ j \in C_{T_{n-1}}^i $ then $$ P_{n-1}^j = x_{T_n-T_{n-1}}^{j } (P_{n}^{j }) $$ 
\item If $k:= I_{n-1} \in \partial_{ext} (C^i_{T_{n-1}}), $ we have  $$ P_{n-1}^{l } = x_{T_n-T_{n-1}}^{l } (P_{n}^{{l}}) + W_{ k \to { l } } $$   where there exists  $ l $  such that $ k\to l \in C^i_{T_{n-1}} $
\item  If $ (I_{n-1}, T_{n-1} ) \in (\bar{\mathcal{C}}^i \times \mathbb{R}_+) \setminus \mathcal{S}^i $ then: \\
- We decide according to the probabilities $$ p= \frac{\beta ( x_{T_n- T_{n-1}}^{I_{n-1}} (P_n^{I_{n-1}})) - \beta_*}{\beta^*-\beta_*} $$ to accept the presence of a spike of neuron $I_{n-1}$.\\

\item[] We update $$ \mathcal{S}^i \gets \mathcal{S}^i \cup \{ (I_m, T_m)\in \bar{\mathcal{C}}^i \times \mathbb{R}_+, C_{T_m}^{I_m}  \subset \mathcal{S}^i\} $$  and go back to step 2. \\
- Else with the probabilities $1-p $ we reject the presence of a spike of neuron $I_{n-1} $ and  $P_{n-1}^{I_{n-1}} = x_{T_n-T_{n-1}}^{I_{n-1} } (P_{n}^{I_{n-1} }). $

We consider all the elements of $  \mathcal{S}^i $ and we always start with the last element to get out of the clan. The update of $ \mathcal{S}^i $ allows us to start the procedure again.

\end{enumerate}
We stop the procedure when all the elements of $\bar{\mathcal{C}}^i$ are filled.

\begin{remark}
 For any site $ (i,t) \in \mathbb{Z} \times \mathbb{R}_+, $ $ C_t^i $ is a Markov jump process taking values in the finite subset of  $\mathbb{Z} $ (see Galves et al. \cite{GLO}) and its infinitesimal generator is given by 
\begin{multline*}
A^{clan}g(C) = \sum_{j \in C}\beta_*[g(C \setminus \{j\}) -g(C)] +  \sum_{j \in \partial_{ext}(C)} (\beta^* - \beta_*)[g(C \cup \{j\}) - g(C)]
\end{multline*}
where $g $ is a test function.
\end{remark}

\begin{proposition}\label{mort-naissance}
Let $d_j = \min_{x^j} \beta_j (x^j), \text{ } d^j = \max_{x^j} \beta_j (x^j) $ and $ b_j =  \sum_{k \to j} (d^k -d_k). $ If $b_j < d_j $ for all $j $ then  $T^i_{stop} $ is finite almost surely i.e. the process is subcritical.
\end{proposition}

\begin{proof}
We shall construct a process $Z_n $ such that for all $n, $ 
$ |  C_{T_n}^i | \leq Z_n $ and such that $Z_n $ evolves as follows:  with probability  $ \frac {\sum_{j \in {Z_n}} d_j}{\sum_{j \in {Z_n}} (b_j +d_j)} $ we have $Z_{n+1} = Z_n -1 $ and with probability $ \frac{\sum_{j \in  Z_n} b_j}{{\sum_{j \in Z_n} (b_j +d_j})} $ we have $Z_{n+1}  =  Z_n +1. $

\end{proof}

In this general case where a neuron has a finite number of neighbors (more than two neighbors) with which it interacts, we can say no more than proposition \ref{mort-naissance}.  Thus, in the following, we put ourselves in the case where each neuron $ i $ has exactly two neighbors so that the neuron $ i $ interacts only with the neurons $ i+1 $ and $ i-1. $ In other words, the incoming neighborhood of $ i $ is $\mathcal{V}_{. \to i} = \{i+1, i-1\}. $ 

\textbf{Algorithm (backward procedure)}

\begin{enumerate}
\item We simulate , $\forall \; l \in \mathbb{Z}, \;\; N_t^{l,s} \text{ and } N_t^{l,p} $ two Poisson processes with respective intensities $\beta_* $ and $\beta^*-\beta_*. $ The jump times of $ N_t^{l,s} $ and $N_t^{l,p} $ are respectively $T_n^{l,s} $ and $T_n^{l,p} $ for the neuron $l $ after $n $ jumps. The jump times $T_n^{l,s} $ will be considered as times of sure jumps (counted by the process $ N_t^{l,s} $) and the jump times $T_n^{l,p} $ will be considered as times of possible jumps (counted by the process $ N_t^{l,p} $)\\

\item Let $ i \in \mathbb{Z} $ fix and   $ T_1= \inf \{ T_1^{i \pm 1, s}, T_1^{i \pm 1, p},\;  T_1^{i,s} \}. $ We set $I_1= i \pm 1$

- If $T_1= T_1^{i \pm 1,s}, $ we set $C_{T_1}^i = i $ and we  put $I_1= i \pm 1 .$\\

- If $T_1= T_1^{i \pm 1,p}  , $ we set $C_{T_1}^i  = \{ i, i \pm 1\}. $ We put $I_1= i \pm 1$\\

- If $T_1= T_1^{i,s}, $ we set $C_{T_1}^i = \emptyset $ and we stop the algorithm. We put $I_1= i .$\\

\item Suppose $T_n $ is the $n-$th jump time of $C_{T_n}^i. $ 
 We have: $$ T_{n+1}=\inf\{T_m^{j, s},  T_m^{j, p}>T_n: |j-C_{T_n}^i|\leq 1, \;  T_m^{k,s} > T_n, \; k\in  C_{T_n}^i \}. $$
- If $T_{n+1}= T_m^{j,p} $ we set: \begin{equation*} 
\begin{cases}
\text{If }j\in C_{T_n}^i, \;\; C_{T_{n+1}}^i=C_{T_n}^i \\ \text{If }j \notin C_{T_n}^i, \;\; C_{T_{n+1}}^i=C_{T_n}^i \cup \{j\}
\end{cases}
\end{equation*}
- If $T_{n+1}= T_m^{k,s} $ we set: \begin{equation*} 
\begin{cases}
\text{If } k\in C_{T_n}^i, \;\; C_{T_{n+1}}^i=C_{T_n}^i \setminus \{k\} \\ \text{If } k \notin C_{T_n}^i, \;\; C_{T_{n+1}}^i=C_{T_n}^i 
\end{cases}
\end{equation*}
We update $C_t^i $ and start the procedure again. We stop the procedure at time $T_{stop}^i $ where $ T_{stop} ^i= \inf\{ t : C_t^i = \emptyset\}.$
\end{enumerate}

Indeed, the whole procedure makes sense only if $ T_{stop} ^i < \infty $ almost surely.

%{\color{red} A enlever cette remarque de la these et remettre l'algo forward}
\begin{remark}
The forward procedure is the same as in the first case where each neuron has a finite number of neighbors.
\end{remark}

The following theorem gives conditions on the finitude of the extinction time.

\begin{theorem} \label{theo:val critique}
 We set $\delta=\frac{\beta_* }{\beta^*-\beta_*}. $
There exists a critical value $0<\delta_c < \infty $ such that:
\begin{enumerate}
\item[-] if $\delta > \delta_c, $ then the extinction time is finite almost surely that is, $ \mathbb{P}( \forall i, \ T_{stop}^i < \infty)=1 $
\item[-] if $\delta < \delta_c, $ then the extinction time is infinite with a positive probability that is, $ \mathbb{P}(\forall i, \ T_{stop}^i = \infty) >0. $
\end{enumerate}
\end{theorem}

\begin{proof}

We first show that $ T^i_{stop} < +\infty $ almost surely for sufficiently large $\delta. $
We observe that we can upper bound $ | C_t^i | $ (where $ | C_t^i | $ is the cardinal of $ C_t^i $ ) by $ Z_t $ almost surely for all $t \geq  0 $ where $ Z_0=1 $ and $ (Z_t)_{t \geq 0} $ is a branching process. With a rate $ n(\beta^* - \beta_*) $ the transition from $ Z_t $ is from $ n $ to $ n+1 $ and with a rate $ n \beta_* $ this transition is from $ n $ to $ n-1. $ 

We can therefore define for any bounded test function $ f, $ the associated infinitesimal generator of $ (Z_t)_{t\geq 0} $ as follows : $$Af(n) = n [ (\beta^*-\beta_*) \left(  f(n+1)-f(n) \right) + \beta_* \left( f(n-1) - f(n) \right) ]. $$
Take $ f(n) = n, $ we obtain : $$ Af = f [(\beta^* - \beta_* ) - \beta_* ] = f (\beta^* - \beta_*)(1- \delta). $$ Then, for $\delta >1, $ we have $ Af (n) = - c f(n) $ where $ -c= (\beta^* - \beta_*)(1- \delta). $ Assuming $ x_t =  \mathbb{E} ( f(Z_t) ) $ and using the Itô formula, we have: $$ x_t =  x_0 + \mathbb{E} \int_0^t Af(Z_s) ds = x_0 -c \int_0^t x_s ds = x_0 e^{-ct}. $$ Therefore, when $t \to \infty, $ we have $ x_t \to x_0 = 1. $ Which implies that if $\delta > 1, \, \mathbb{P} (T^i_{stop} < \infty ) \geq \mathbb{P} (\lim_{t \to \infty} Z_t = 0) = 1 $ thus ensuring that $\delta_c < 1. $ \\

We now show that for all $\delta<\delta_c ,  \, T^i_{stop} = +\infty $ with positive probability.

For this proof, we will use the classical graphical construction of $C^i_t$ (see Ferrari et al \cite{FGGL}, Griffeath \cite{Gri}).  We work within the space-time diagram $\mathbb{Z} \times [0,\infty[. $ For each $i \in \mathbb{Z}, $ we consider $ N_t^{i,s} \text{ and } N_t^{i,p}$ two independent Poisson processes with respective intensities $\beta_* $ and $\beta^*-\beta_*. $ The jump times of $ N_t^{i,s} $ and $N_t^{i,p} $ are respectively $T_n^{i,s} $ and $T_n^{i,p} $ for the neuron $i $ after $n $ jumps.

For each $i \in \mathbb{Z}, $ we draw graphical sequences as follows.
First draw arrows pointing from $(i-1, T_n^{i,p}) $ to $(i, T_n^{i,p}) $ and from $(i+1, T_n^{i,p}) $ to $(i, T_n^{i,p}) $ for all $n\geq 1, \, i\in \mathbb{Z}. $ Second, $\delta $'s at all $(i, T_n^{i,s}), $ for all $n\geq 1, \, i \in \mathbb{Z}. $ We also suppose that time is going up which implies that we thus obtain a random graph $\mathcal{P}. $ Let us say that there is a chain of vertical upward and horizontal directed edges in the random graph that leads from $(i, 0) $ to $ (j,t) $ ( with $ j \in \{i+1, i-1\} $) without passing through a $\delta .$ 
Notice that $C_t^i $ is the set of the clan of ancestors of site $(i, t), $ that is $$ C_t^i = \{ j: \text{ there is a path from } (i,0) \text{ to } (j,t) \text{ for } j=i \pm 1\}. $$
It is obvious to notice that $T^i_{stop} < \infty $ if and only if $\bar{\mathcal{C} }^i= \cup_{t\geq 0} C_t^i $ is a finite set. We will therefore show that $ \mathbb{P}(T^i_{stop} < \infty) = \mathbb{P}(|\bar{\mathcal{C}}^i| < \infty) < 1 $ for sufficiently  small values of $\delta $ using  classical contour techniques. (see Griffeath \cite{Gri}.)

For this, on $|\bar{\mathcal{C}}^i| < \infty, $ we draw the contour of $\bar{\mathcal{C}}^i $ as follow. 

Starting from $ (i-\frac{1}{2},0). $ Let $\Gamma $ be a possible path of the graph $\mathcal{P}. $ $\Gamma $  consists of $ 4n $  alternating vertical and horizontal edges for some $n \geq 1 $ which we encode as a succession of direction vectors $ (D_1, \cdots , D_{2n} ). $  Each of the $ D_i $ can be one the seven triplets $$ dld, drd, dru, ulu, uru, urd, dlu, $$ where $d,  u,  l \text{ and }  r $ stand for down, up, left and right, respectively. Note that $ uld $ cannot occur in a possible path $\Gamma $ because the direction of $ uld $ is counter-clockwise.
 We  start at $( i-\frac{1}{2},0) $ and move clockwise around the curve.

The two figures below show examples of possible paths for $n = 3 $ and $n= 4. $ Figure.1 shows a possible path with $ n=3 $  and in this case we have $$ \Gamma : ulu, ulu, urd, drd, drd, dlu. $$ For $ n=4, $ Figure.2 gives  $$ \Gamma : ulu, ulu, urd, drd, dru, urd, dld, dlu. $$ 

\begin{center}
\includegraphics[scale=0.35]{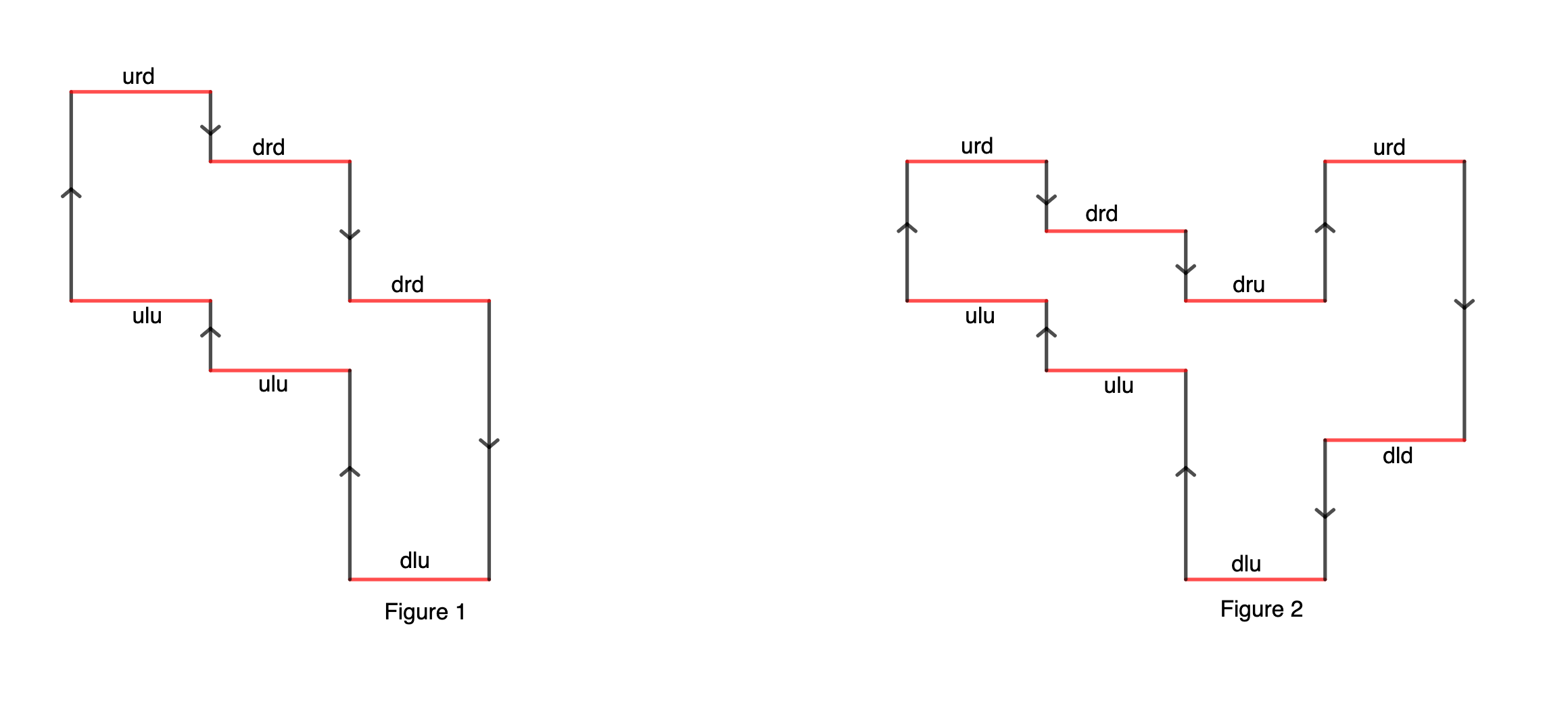}
\end{center}

Writing $ N(dld), N(drd), \cdots $ for the number of appearances of the different direction vectors, we have that $ N(dlu) = 1 $ ($ dlu $ is the last triplet of which appears exactly one single time) and $$ N(dru)= N(urd)-1 \leq n/2,  \  N(drd) + N(dru) + N(uru) + N(urd) =n  .$$ (for more details, see Ferrari et al. \cite{FGGL}.)

We first observe that the occurrence of either $uru, \, urd, \text{ or } drd $ can be upper bounded by $\delta. $ This is due the fact that the probability associated with $ uru \text{ or } drd $ is $\frac{\delta}{1+ 2\delta}$ and that of $ urd $ is $\frac{\delta}{2+ \delta}$. In the same way, we observe that the occurrence of either $ dld, \, ulu \text{ or } dlu $ can be upper bounded by $1. $ Indeed, the associated probability with its directions is $\frac{1}{1+ \delta}.$ Therefore we obtain the following list of upper bounds
\begin{eqnarray*}
uru \text{ occurs with probability at most }  \delta \\
urd \text{ occurs with probability at most }  \delta \\
drd \text{ occurs with probability at most }  \delta \\
dru \text{ occurs with probability at most } 1 \\
dld \text{ occurs with probability at most } 1 \\
ulu \text{ occurs with probability at most }  1 \\
dlu \text{ occurs with probability at most }  1.
\end{eqnarray*}

In the above list, we have upper bounded the probability  associated with $dru $ which is given by $\frac{\delta}{3 \delta}=\frac{1}{3}, $ by 1. \\
For a given contour having $ 4n $ edges, with $ n \geq 2 $, its probability is therefore upper bounded by $$ \delta^{ N(drd)+N(uru) +N(urd)}= \delta^ {n-N(dru)} \leq \delta^{n/2}. $$

Indeed, for each triplet we have $4 $ possible choices. The first entry of a given triplet is always fixed by the previous triplet in the sequence, and for the first triplet $D_1 $ the first entry is always $u.$

Then, for $n = 1, $  the probability of appearance of a contour of length $4$ is equal to $\mathbb{P}(D_1=urd)= \frac{\delta}{2+ \delta} \leq \delta.$\\
We also have, for  $ n = 2, $ the probability of appearance of a contour of length $8$ is equal to 
\begin{multline*}
 \mathbb{P}(D_1 = ulu, \ D_2 = urd, \ D_3 = drd) + \mathbb{P}(D_1 = ulu, \ D_2 = uru, \ D_3 = urd) + \\
  \mathbb{P}(D_1 = uru, \ D_2 = urd, \ D_3 = dld) + \mathbb{P}(D_1 = urd, \ D_2 = drd, \ D_3 = dld) \leq 4 \delta^2. 
 \end{multline*}

\begin{remark}
In the above probabilities, we have not put the direction $ D_4 = dlu $ because it is a certain direction. It is common to all possible paths and its probability of occurrence is 1.
\end{remark}

Therefore, a very approximate upper bound on the total number of possible triplets $ (D_1, \cdots , D_{2n}) $ is given by $ 4^{2n} = 16^n. $ We get for all $\delta < \frac{1}{(16)^2}, $
$$ \mathbb{P}(T^i_{stop} < \infty ) \leq \delta + 4 \delta^2 + \sum_{n \geq 3}(16 )^n \delta^{n/2} =  \delta + 4 \delta^2 +  \frac{(16 \sqrt{\delta})^3}{1-16 \sqrt{\delta}}.  $$

We set $\phi: \delta \mapsto \phi(\delta) = \delta + 4 \delta^2 + \frac{(16 \sqrt{\delta})^3}{1-16 \sqrt{\delta}}. $ Then, $  \mathbb{P}(T^i_{stop} < \infty ) \leq \phi(\delta). $ \\
As $\delta \to 0, \,  \phi(\delta) \to 0 $ which implies that there exists $ \delta_c $  such that $ \phi(\delta_c) = 1. $ As a consequence, $  \mathbb{P}(T^i_{stop} < \infty ) <1, \, \forall \, 0 < \delta < \delta_c. $

We therefore conclude that $ \delta_c $ exists and $0 < \delta_c <1.$

\end{proof}

\subsection{Some simulations}

We simulate the state $X_0(i) $ in the stationary regime for a fixed neuron $i \in \mathbb{Z} $ at time $0 $ and estimate its density. The main purpose of this simulation is to have an idea about the theoretical distribution of $X_0(i) $ in its stationary regime and whether this distribution is impacted by the specification of $F^i.$ 

We denote by $\mathcal{D} $ the set of neurons which belong to a clan of ancestors of neuron $i $ at a time $t$ or to its neighborhood.

To do this, we apply the following algorithm:

\begin{enumerate}
\item Initialize the family $ \mathcal{V}_{. \to i} $ of non empty neighborhoods of the neuron $i$\\
\item Initialize  $C_0^i =i $ the clan of ancestors  of neuron $i$ at time $t=0. $ \\
\item  For all time $t >0$ we let $C_t^i $ the clan of ancestors of neuron $i$ at time $t$\\
\item While $ | C_t^i | >0 $ (where $| C_t^i| $ denotes the cardinality of $C_t^i$) do 

 -Determine the next jump time $ t_{next} > t $ in the clan of ancestors of neuron $i $ at time $t_{next} $ and in $ \partial_{ext} (clan) $, the correspondant neuron $j $ and the nature of jump
 
- If neuron $j \in C_t^i $ and the jump is sure, then  $C_{t_{next}}^i = C_t^i \setminus \{i\} $ 
 
- If $j \in  C_t^i $ and the jump is possible $C_{t_{next}}^i = C_t^i $ 
 
- If $j \in V(C_t^i) $ (where $V(C_t^i) := \cup_{j \in C_t^i} \mathcal{V}_{. \to j} $) and the jump is sure, then $C_{t_{next}}^i = C_t^i $ 
 
- If $j \in V(C_t^i) $ and the jump is possible $C_{t_{next}}^i = C_t^i \cup \{j\} $
 
- We update $ t \leftarrow t_{next}$

end While.\\

\item We determine the chronological list of the different jump times from $0 $ to the last time which makes the clan empty. 

- For each of these jump times, we indicate the associated neuron and the nature of the jump.

- If the jump is sure, we simulate a random state following a distribution $F^i$ at the neuron associated with this jump time. \\

\item We set $m = \infty$. While $m >0 $ do

- Let $m $ be the rank of the last possible jump time $T_m $ of $\mathcal{D}$  in the chronology of jump times. Let $k $ be the neuron associated with this jump.\\

\item We determine the rank $r $ of the last certain jump time $T_r > T_m $ of $k$ in the chronology of jump times.  The state of $k$ is determined recursively from its state at time $T_r$ to its state at time $T_m$ as follows:

- For $s  \in \{1 ,\cdots, r-m-1\}$ let  \text{ $ x=$ state of $k$  at time $T_{r-s+1} $}  . 

-  Let $dt = T_{r-s+1} - T_{r-s} $ and $j$ the neuron associated with the jump time $T_{r-s}. $ The state of $k$ at time $T_{r-s} $ is $ x_{-dt}^k(x)+W_{j \to k}*1_{\{\text{sure  jump  of  $j $  at $ T_{r-s}$} \} } $ with $ W_{j \to k} $ the inhibition weight of $j $ on $k$. 

- We determine rather the occurence is effective or not of the jump of $k$ at time $T_m$ thanks to its state at time $T_{m+1}$. \footnote{The jump occurs with a Bernoulli distribution with parameter $(\beta(x_{-dt}^k(x)) -\beta_*)/(\beta^*-\beta_*)$}

- If the jump is effective, we simulate a random state for $k$ at time $T_m$  following a distribution $F^i$. Otherwise, we determine the state of $k $ at time $T_m $ as $x_{T_{m}-T_{m+1}}(x) $ where \text{$x=$ state of $k$ at time $T_{m+1}.$}
 Let $m$ be the new rank of the last possible jump time of $\mathcal{D}$ and repeat the procedure.
 
end While.

\begin{remark}
After this step, we know the exact nature of all jumps.
\end{remark}

\item Determine for neuron $i$ its first safe jump time $T_n$ where $n$ is the rank of this time in the chronology of jump times.\\

\item The state of neuron $i$ is determined recursively from its state at time $T_n$ to $T_0$ as follows:

- For $s \in \{1, \cdots, n-1 \} $ let \text{$ x=$ state of neuron $i$  at time $ T_{n-s+1}.$ }

-  Let $dt = T_{n-s+1} - T_{n-s} $ and $j$ the neuron associated with the jump time $T_{n-s}. $ The state of neuron $i$ at time $T_{n-s} $ is $ x_{-dt}^i(x)+W_{j \to i}*1_{\{sure\  jump\ of \ j \ at \  T_{r-i} \}} $ with $ W_{j \to i} $ the inhibition weight of $j $ on $i$. 

\begin{remark}
The last value determined is the initial state of neuron $i.$
\end{remark}
\end{enumerate}

In the three following examples we consider $\alpha_i(x)=x,$ $\beta_i(x)= 3+\mathbf{1}_{\{x \le 2\}}, $ $W_{i \to j}=1.$ To verify if the distribution of inhibition state depends on the distribution $F^i, $ we consider three different distributions for $F^i $ that are $ \mathcal{E}(1), $  $ \mathcal{E}(10) $ and $ 0.5 \delta_1+0.5 \delta_2. $  We simulate, with the algorithm described above $N=1000 $ values for the inhibition state . We then estimate non parametrically the distribution of the inhibition state in these three cases of distribution $F^i $ and we compare them.

The stationary distribution of the process in the three following cases seems to be continuous. We do not provide a proof here, this is outside the scope of this paper.

\begin{figure}[h]
    \begin{minipage}[c]{.46\linewidth}
        \centering
        \includegraphics[scale=0.4, height=6cm, width=7cm]{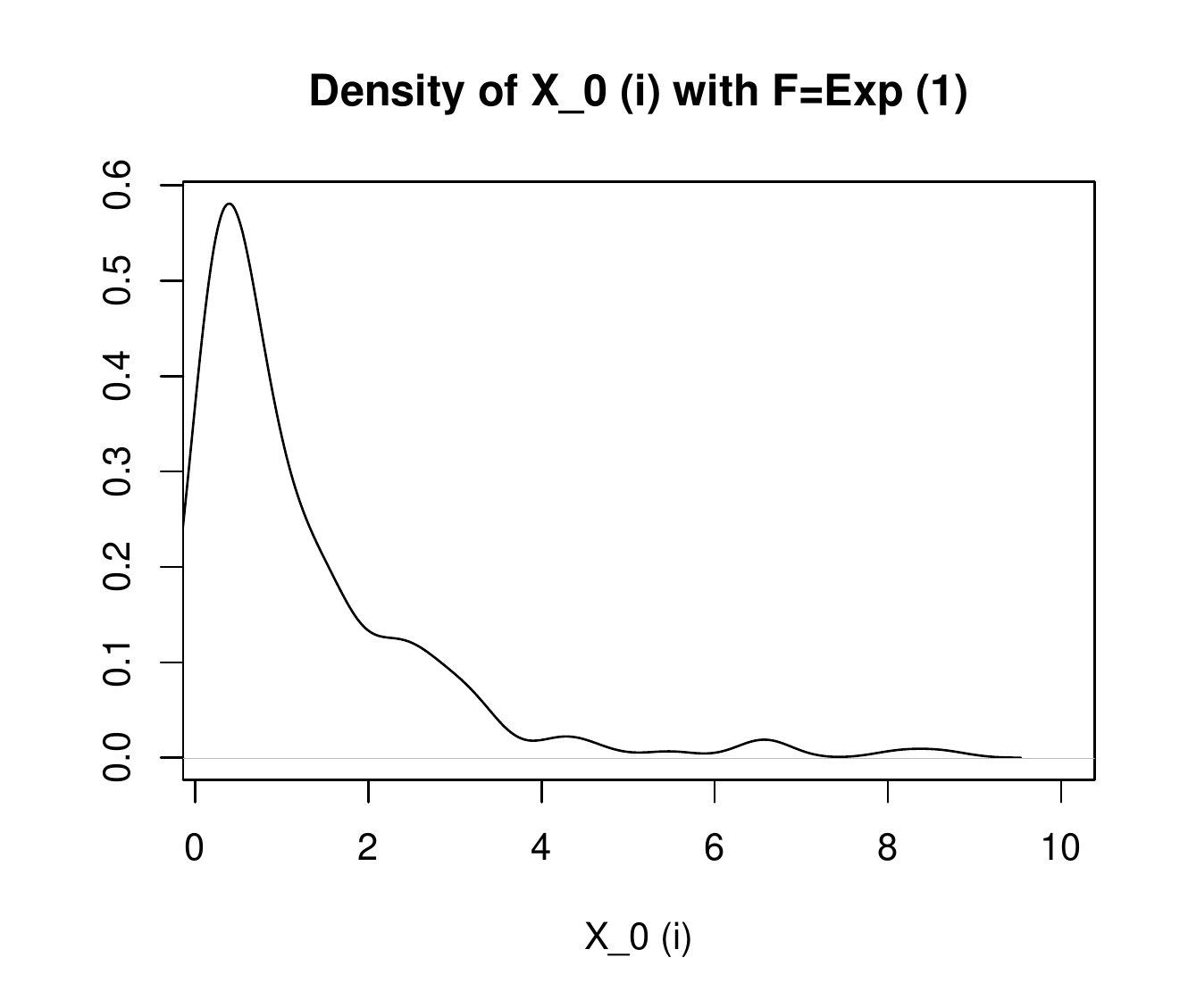}
        \caption{}
    \end{minipage}
    \hfill%
    \begin{minipage}[c]{.46\linewidth}
        \centering
        \includegraphics[scale=0.4, height=6cm, width=7cm]{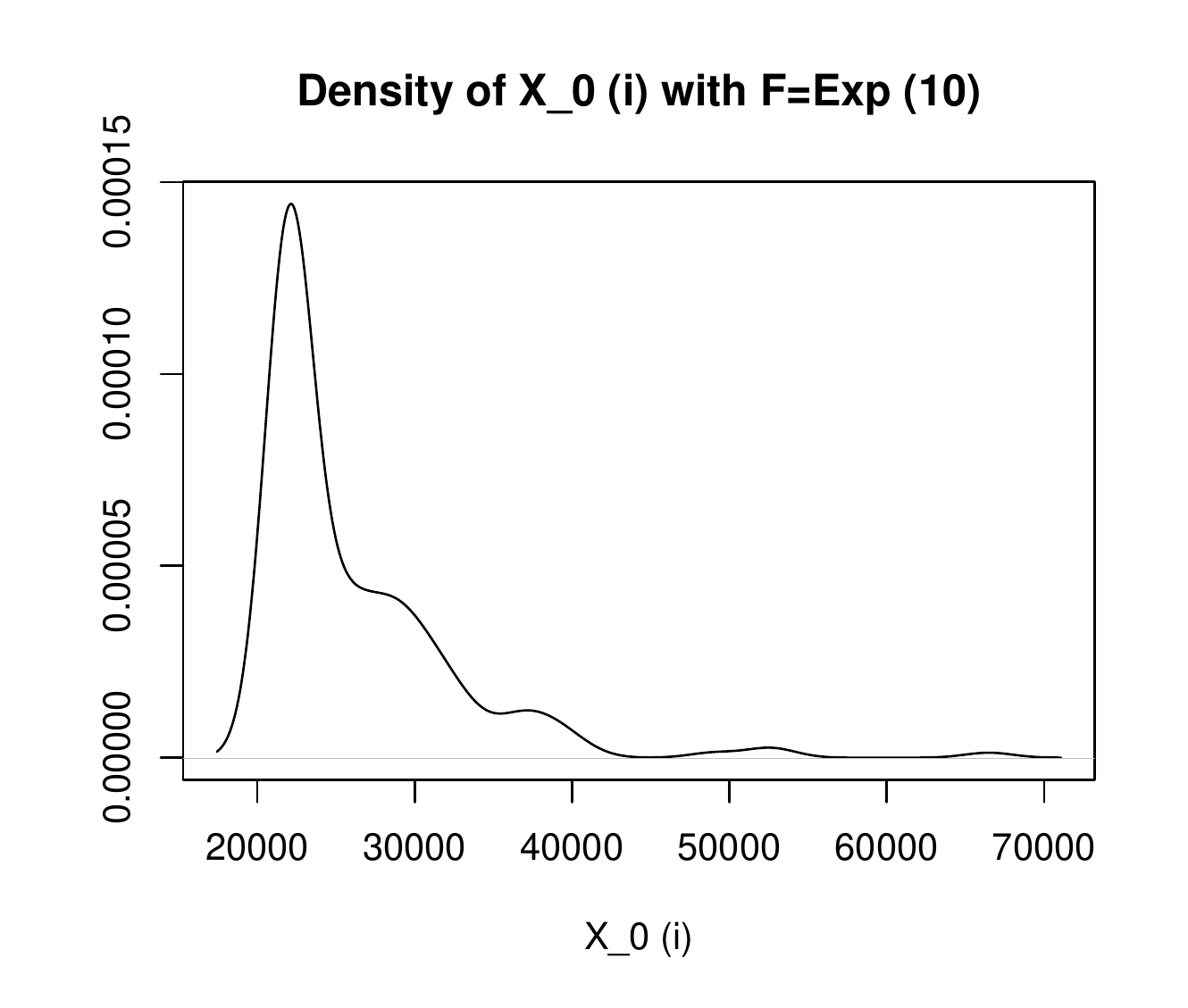}
        \caption{}
    \end{minipage}
\end{figure}

We can remark that the distribution of state of inhibition $X_0(i) $ in stationary regime is concentrated in the interval $(0,4) $ when $F^i = \mathcal{E}(1) $ whereas this distribution is rather concentrated on the interval $(20000, 40000) $ when $F^i = \mathcal{E}(10). $ This shows that these two distributions of state $X_0(i) $ are different.

\begin{center}
	\includegraphics[scale=0.52]{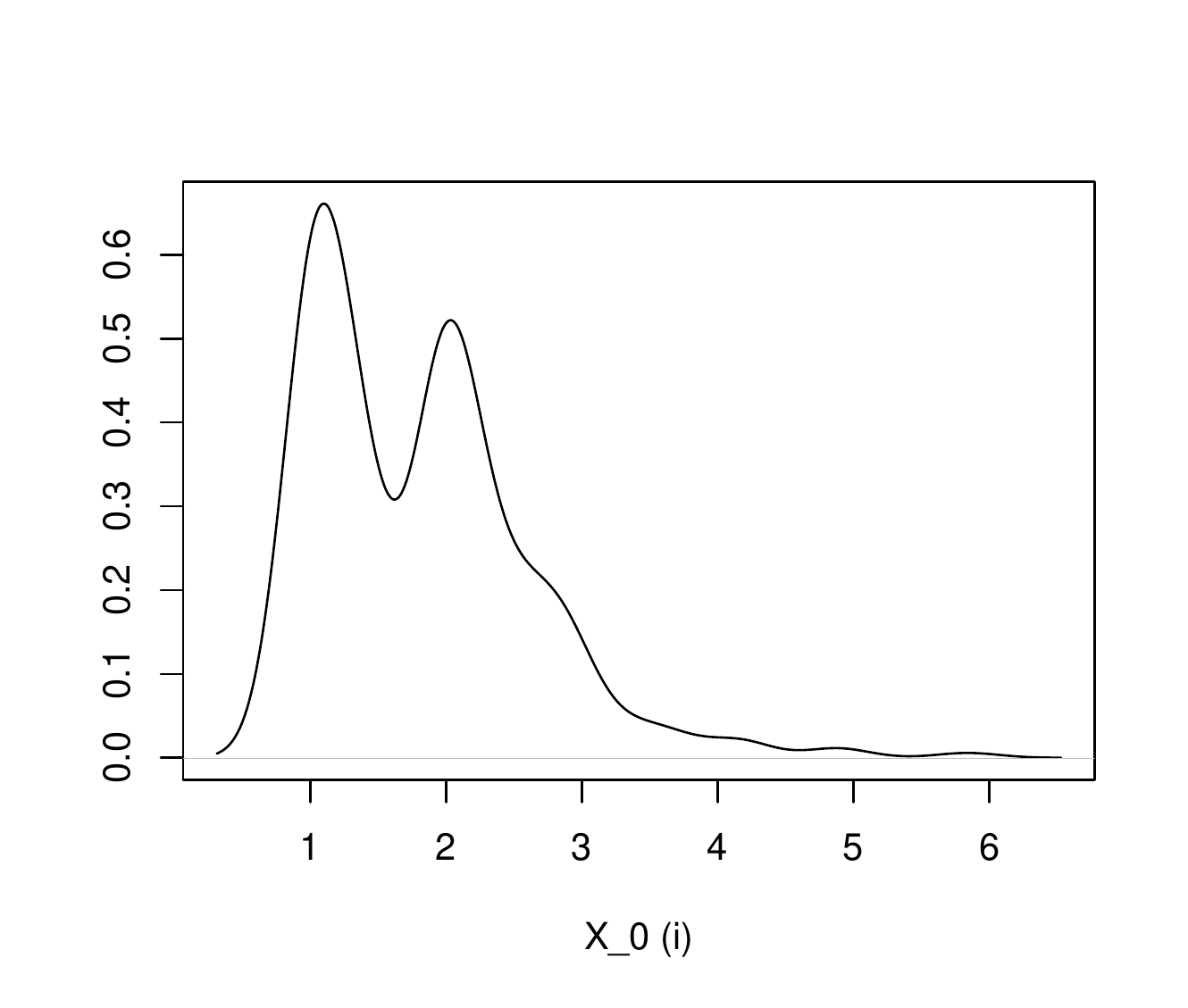}
\end{center}

In this example, the distribution of the state of inhibition $X_0(i) $ in stationary regime seems to be continuous although $F^i $ is discrete.  We do not provide a proof here, this is outside the scope of this paper.  We observe two local extrema at $1$ and $2 $ which are linked to the jumps because of the Dirac. These extrema suggest that jumps are very frequent in this process.
 
\textbf{Acknowledgments:}\\
The author thanks Eva L\"{o}cherbach for the many discussions that led to this version of the paper. This research was conducted within the part of the Labex MME-DII(ANR11-LBX-0023-01)  project and the CY Initiative of Excellence (grant "Investissements d'Avenir" ANR-16-IDEX-0008), Project EcoDep PSI-AAP 202-00000013


\begin{thebibliography}{99}


\bibitem{BLMZ} B\'{e}naïm, M ., Le Borgne, S., Malrieu, F., Zitt, P.: Qualitative properties of certain piecewise deterministic Markov processes. Ann. Inst. H. Poincar\'{e} Probab. Statist. 51 (3) 1040-1075, August 2015

\bibitem{CFF} Comets, F., Fernandez, R., Ferrari, P. A. (2002). Processes with long memory: Regenerative construction and perfect simulation. Ann. of Appl. Probab., 12, No 3, 921–943.

\bibitem{Cot} Cottrell, M.: Mathematical analysis of a neural network
with inhibitory coupling. Stochastic Processes and their Applications 40 (1992) 103-126

\bibitem{FGGL}  Ferrari, P. A., Galves, A.,   Grigorescu, I.,   L\"{o}cherbach, E.: Phase Transition for Infinite Systems of Spiking Neurons. Journal of Statistical Physics (2018) 172:1564–1575 DOI 10.1007/s10955-018-2118-6

%\bibitem{FMMN } Ferrari, P. A., Maass, A., Martinez, S., Ney, P. (2000). Cesàro mean distribution of group automata starting from measures with summable decay. Ergodic Theory Dyn. Syst., 20(6), 1657– 1670.

\bibitem{FRST} Fricker C., Robert P., Saada E., Tibi D. (1993) Analysis of a Network Model.  Cellular Automata and Cooperative Systems. NATO ASI Series (Series C: Mathematical and Physical Sciences), vol 396. Springer, Dordrecht.  

\bibitem{GL} Galves, A. and Löcherbach, E. (2013). Infinite systems of interacting chains with memory of variable length - a stochastic model for biological neural nets. Journal of Statistical Physics 151 896–921.

\bibitem{GLO} Galves, A., L\"{o}cherbach, E., Orlandi, E.: Perfect simulation of infinite range Gibbs measures and coupling with their finite range approximations. J Stat Phys DOI 10.1007/s10955-009-9881-3 (2009)

%\bibitem{GKNP} W. Gerstner, W. M. Kistler, R. Naud, and L. Paninski. Neuronal Dynamics: From Single Neurons to Networks and Models of Cognition. Cambridge University Press, Cambridge, 2014.

\bibitem{GHL} Goncalves, B., Huillet, T., L\"{o}cherbach, E.: On decay-surge population models. (2020) arXiv:2012.00716

\bibitem{Gri} Griffeath, D.: The basic contact process. Stoch. Proc. Appl. 11, 151-185 (1981)

\bibitem{MT} Meyn, Sean., R,L Tweedie.: Stability of Markovian processes III: Foster- Lyapunov criteria for continuous-time processes. Adv. Appl. Prob. 25, 518-548 (1993) Printed in N. Ireland 

\end{thebibliography}
\end{document}